\documentclass[reqno,10pt, centertags]{amsart}
\usepackage{amssymb,upref,esint,color}
\usepackage{hyperref}
\usepackage{etoolbox}
\newcommand*{\mailto}[1]{\href{mailto:#1}{\nolinkurl{#1}}}

\makeatletter
\usepackage{verbatim}
\usepackage{amscd}
\usepackage{framed}
\numberwithin{equation}{section}
\input epsf

\newtheorem{theorem}{Theorem}[section]
\newtheorem{definition}[theorem]{Definition}
\newtheorem{lemma}[theorem]{Lemma}

\newtheorem{proposition}[theorem]{Proposition}
\numberwithin{equation}{section}
\theoremstyle{definition}

\begin{document}

\title[Self-adjoint extensions with compact resolvent]
{Self-adjoint extensions with compact resolvent}

\author[Yicao Wang]{Yicao Wang}
\address
{School of Mathematics, Hohai University, Nanjing 210098, China}

\baselineskip= 20pt
\begin{abstract}Let $T$ be a densely defined closed symmetric operator with equal deficiency indices in a separable complex Hilbert space $H$. In this paper, we prove that $T$ has a self-adjoint extension with compact resolvent if and only if the domain $D(T)$ of $T$ is compactly embedded in $H$ w.r.t. the graph norm on $D(T)$. If it is the case, we also prove that all self-adjoint extensions with compact resolvent can be parameterized by unitary operators $U$ on a certain Hilbert space such that $U-Id$ is compact.
\end{abstract}
\maketitle


\tableofcontents

\section{Introduction}
Von Neumann's theory of self-adjoint extensions of symmetric operators is now a classical theory in functional analysis. Given a densely defined closed symmetric operator $T$ in a separable Hilbert space $H$, according to von Neumann, if the deficiency indices of $T$ are $(n_+,n_-)$ where $n_\pm\in \mathbb{N}\cup \{\infty\}$, then there is a self-adjoint extension if and only if $n_+=n_-$. Once a self-adjoint extension of $T$ does exist, there are infinitely many--all of them are parameterized by unitary maps from the positive deficiency subspace onto the negative deficiency subspace.

In this paper we always assume the deficiency indices of $T$ satisfy $n_+=n_-=n$. Thus self-adjoint extensions exist. We say a self-adjoint extension $\tilde{T}$ of $T$ has compact resolvent if $(\tilde{T}-\mathfrak{i})^{-1}$ is compact\footnote{If $(\tilde{T}-\mathfrak{i})^{-1}$ is compact, so is $(\tilde{T}-\lambda)^{-1}$ for any $\lambda\in \rho(\tilde{T})$, where $\rho(\tilde{T})$ is the resolvent set of $\tilde{T}$.}. The property of having compact resolvent means $\tilde{T}$ is regular in a certain sense, e.g., $\tilde{T}$ has only purely discrete spectrum. Many classical symmetric operators have self-adjoint extensions with compact resolvent, e.g., regular 1-dimensional Schrodinger operators on a finite interval and the Laplacian $\triangle$ on a bounded smooth domain in $\mathbb{R}^m$. However, on the abstract level two natural and basic questions remain: (1) among those infinitely many self-adjoint extensions, is there one with compact resolvent? and (2) if it is the case, how many such extensions are there?

As for the first question, it is easy to see that for a self-adjoint extension with compact resolvent to exist the domain $D(T)$ (equipped with the graph norm induced from $T$) of $T$ has to be \emph{compactly embedded} into $H$ (see Lemma \ref{nece}). Then is this necessary condition also sufficient? For a finite $n$, this is really the case, but for $n=\infty$, as far as we know a complete answer may not yet exist. The best result we can find is Theorem 5 on Page 224 of \cite{birman1986spectral}, implying that if the form domain of the quadratic form associated to the Friedrichs extension $T_F$ of a semi-bounded symmetric operator is compactly embedded in $H$, then $T_F$ has compact resolvent. Some authors still take the existence of self-adjoint extensions with compact resolvent as a basic assumption of their theorems, e.g., \cite[Thm.~5.1]{booss1998maslov}, \cite[Thm.~4.3]{behrndt2022construction}. Our first main theorem is to provide the following simple criterion for the existence of self-adjoint extensions with compact resolvent in terms of $T$ itself only:
\begin{theorem}\label{Th1}For a densely defined closed symmetric operator $T$ in $H$ with equal deficiency indices, $T$ has a self-adjoint extension with compact resolvent if and only if $D(T)$ equipped with its graph norm is compactly embedded in $H$.
\end{theorem}
The necessity part is easy and seems to be well-known. The difficulty lies in the sufficiency part. The proof given in \S~\ref{existence} depends on some topological observations concerning the Cayley transform and a formula concerning Weyl function associated to a boundary triplet.

Recall that a boundary triplet for $T$ is a triple $(G, \Gamma_0, \Gamma_1)$ consisting of a Hilbert space $G$ and two linear maps $\Gamma_0, \Gamma_1: D(T^*)\rightarrow G$ such that\\
(1) The map $(\Gamma_0, \Gamma_1):D(T^*)\rightarrow G\times G$ is surjective;\\
(2) For any $x, y\in D(T^*)$, the following abstract Green's formula holds:
\[(T^*x, y)_H-(x,T^*y)_H=(\Gamma_1x, \Gamma_0y)_G-(\Gamma_0 x, \Gamma_1y)_G.\]
For a symmetric operator $T$ with equal deficiency indices, a boundary triplet always exists. In particular, the extensions $T^*|_{\Gamma_0^{-1}(0)}$ and $T^*|_{\Gamma_1^{-1}(0)}$ are self-adjoint. With the aid of a boundary triplet, there is a one-to-one correspondence between self-adjoint extensions of $T$ and unitary operators on $G$. More precisely, the domain of the self-adjoint extension associated to the unitary operator $U$ is
\begin{equation}\{x\in D(T^*)|(Id-U)\Gamma_1x=\mathfrak{i}(Id+U)\Gamma_0x\},\label{sau}\end{equation}
where $Id$ is the identity operator on $G$. That's to say the set of self-adjoint extensions is parameterized by the unitary group $\mathbb{U}(G)$ on $G$. This is an easy conclusion of \cite[Prop.~14.4]{schmudgen2012unbounded} and more in the spirit of von Neumann's original vision. However, in the literature there do exist other schemes trying to parameterize self-adjoint extensions, e.g., using Lagrangian subspaces, self-adjoint linear relations and even unbounded operators \cite{schmudgen2012unbounded, behrndt2020boundary, grubb1968characterization}. In our viewpoint, as for the answer to the second question, unitary operators are more suitable parameters. Here comes our second main theorem.
\begin{theorem}\label{thm2}If $(G, \Gamma_0, \Gamma_1)$ is a boundary triplet for $T$ such that the self-adjoint extension $T^*|_{\Gamma_0^{-1}(0)}$ has compact resolvent, then via (\ref{sau}) all self-adjoint extensions with compact resolvent are parameterized by unitary operators $U\in \mathbb{U}(G)$ such that $U-Id$ is compact.
\end{theorem}
The advantage of this parameterization is clear: when the deficiency indices are $(\infty, \infty)$, such unitary operators come together to constitute a closed normal subgroup $\mathbb{U}_\infty(G)$ of $\mathbb{U}(G)$. Note that due to the polar decomposition, $\mathbb{U}_\infty(G)$ is a deformation retract of the larger group $\mathbb{GL}_c(G)$ consisting of invertible operators that are compact perturbations of $Id$. According to \cite{palais1965homotopy}, $\mathbb{U}_\infty(G)$ has the same homotopy type as the stable general linear group $\mathbb{GL}(\infty)$. That's to say our Thm.~\ref{thm2} reflects more clearly the nontrivial global structure of the set of self-adjoint extensions with compact resolvent.

To prove Thm.~\ref{thm2} in \S~\ref{para}, we have to establish a Krein type resolvent formula (Prop.~\ref{res}) and a variant of the boundary triplet $(G, \Gamma_0, \Gamma_1)$ should be used to achieve this. All relevant preliminary lemmas were actually contained in \cite{wang2024complex}, but for the reader's convenience we have included a self-contained presentation.

Here are our conventions on notation. An inner product $(\cdot, \cdot)$ on a Hilbert space $H$ is linear in the first variable and conjugate-linear in the second. The induced norm will be written as $\|\cdot\|$. If necessary, the inner product is often denoted by $(\cdot, \cdot)_H$ to emphasize the underlying Hilbert space and to distinguish it from an ordered pair. We use $\oplus$ (resp.~$\oplus_\bot$) to denote a topological (resp.~orthogonal) direct sum. For a subspace $V\subseteq H$, its orthogonal complement is $V^\bot$. The zero vector space will be just denoted by $0$. The identity operator on $H$ will be just denoted by $Id$ if the underlying space is clear from the context. For a constant $c$, $c\times Id$ is often written simply as $c$ if no confusion arises. $D(A)$ is the domain of an operator $A$, and $\textup{ker}A$ (resp. $\textup{Ran}A$) is the kernel (resp. range) of $A$. $A\subseteq B$ means $D(A)\subseteq D(B)$ and $B|_{D(A)}=A$. The (point, continuous, or residual) spectrum of an operator $A$ is denoted by $\sigma(A)$ ($\sigma_p(A)$, $\sigma_c(A)$, or $\sigma_r(A)$) and the resolvent set of $A$ by $\rho(A)$. The set of bounded operators from $H_1$ to $H_2$ is denoted by $\mathbb{B}(H_1,H_2)$ and if $H_1=H_2=H$, $\mathbb{B}(H)$ will be used.

\section{Preliminaries and some lemmas}
This section is devoted to laying the background of the paper and establishing some lemmas to reduce the investigation to the essential case.

Let $H$ be an infinite-dimensional complex separable Hilbert space. Recall that a linear operator $A$ defined in $H$ is closed if the graph of $A$ is a closed subspace of $H\oplus_\bot H$. A densely defined closed symmetric operator $T$ in $H$ is a linear operator from a linear subspace $D(T)\subset H$ to $H$ such that $\overline{D(T)}=H$ and $(Tx,y)=(x,Ty)$ for any $x,y\in D(T)$.

In this paper, when referring to a symmetric operator $T$, we always mean a densely defined closed one. The domain $D(T^*)$ of the adjoint $T^*$ consists of $y\in H$ such that
\[(Tx,y)=(x,z),\quad \forall x\in D(T)\]
for some (unique) $z\in H$, and $T^*y=z$. $T^*$ is always closed and $T\subseteq T^*$. On $D(T^*)$ there is the useful graph inner product
\[(x,y)_g=(x,y)_H+(T^*x,T^*y)_H,\quad \forall x,y\in D(T^*).\]
The induced norm $\|\cdot\|_g$ is called the \emph{graph norm}. $D(T^*)$ with the graph inner product is a Hilbert space and $D(T)$ is then a closed subspace. If furthermore $T=T^*$, we say $T$ is self-adjoint. If $T\subseteq \tilde{T} \subseteq T^*$, $\tilde{T}$ is called an extension of $T$; in particular if $\tilde{T}=\tilde{T}^*$, we say $\tilde{T}$ is a self-adjoint extension. A fact we shall frequently use is that for $\lambda\in \mathbb{C}$ with $\Im \lambda\neq 0$, we have
\[\|(T-\lambda)x\|_H\geq |\Im \lambda|\times \|x\|_H, \quad \forall x\in D(T).\]
As a consequence, $\textup{Ran}(T-\lambda)$ is closed in $H$ and $T-\lambda$ is a topological linear isomorphism between $(D(T), \|\cdot\|_g)$ and $\textup{Ran}(T-\lambda)$.

The number $n_+:=\dim \textup{ker} (T^*-\mathfrak{i})$ (resp.~$n_-:=\dim \textup{ker} (T^*+\mathfrak{i})$) is called the positive (resp.~negative) deficiency index. We have the well-known von Neumann decomposition (w.r.t. the graph inner product).
\[D(T^*)=D(T)\oplus_\bot \textup{ker} (T^*-\mathfrak{i})\oplus_\bot \textup{ker} (T^*+\mathfrak{i}).\]
Self-adjoint extensions exist only for the case $n_+=n_-=n$. We shall assume this is the case throughout the rest of the paper. Let us recall von Neumann's approach to constructing all self-adjoint extensions. The Cayley transform $C(T):=(T-\mathfrak{i})(T+\mathfrak{i})^{-1}$ of $T$ is a \emph{unitary} map from $\textup{Ran}(T+\mathfrak{i})$ to $\textup{Ran}(T-\mathfrak{i})$. Since $(\textup{Ran}(T\pm \mathfrak{i}))^\bot=\textup{ker}(T^*\mp \mathfrak{i})$, one can take a unitary map $V$ from $\textup{ker}(T^*-\mathfrak{i})$ to $\textup{ker}(T^*+ \mathfrak{i})$ and obtain a unitary operator $U_V:=C(T)\oplus_\bot V$ on $H$. The inverse Cayley transform of $U_V$ is a self-adjoint extension of the original $T$ and all self-adjoint extensions of $T$ arise in this manner. For more technical details, see for example \cite[Chap.~13]{rudin1991functional}.

Now the necessity part of Thm.~\ref{Th1} is known and quite easy.
\begin{lemma}\label{nece}If $T$ has a self-adjoint extension with compact resolvent, then $(D(T),\|\cdot\|_g)$ is compactly embedded in $H$.\end{lemma}
\begin{proof}Let $\tilde{T}$ be a self-adjoint extension of $T$ such that $(\tilde{T}-\mathfrak{i})^{-1}$ is compact and let $V$ be the unitary map associated to this extension in von Neumann's construction. Then
\[U_V=(\tilde{T}-\mathfrak{i})(\tilde{T}+\mathfrak{i})^{-1}=Id-2\mathfrak{i}(\tilde{T}+\mathfrak{i})^{-1}.\]
 Let $\{x_n\}\subseteq D(T)$ be bounded w.r.t. the graph norm, i.e., there is a positive constant $C$ such that
 $\|x_n\|^2_H+\|T^*x_n\|^2_H< C$. This is equivalent to that $\{(T+\mathfrak{i})x_n\}$ is bounded in $H$, because
 \begin{eqnarray*}
 ((T+\mathfrak{i})x_n, (T+\mathfrak{i})x_n)_H&=&(x_n,x_n)_H+(Tx_n,Tx_n)_H+\mathfrak{i}[(x,Tx)-(Tx,x)]\\
 &=&\|x_n\|_H^2+\|T^*x\|_H^2.
 \end{eqnarray*} Thus the sequence
\[\{(\tilde{T}+\mathfrak{i})^{-1}(T+\mathfrak{i})x_n=(\tilde{T}+\mathfrak{i})^{-1}(\tilde{T}+\mathfrak{i})x_n=x_n\}\]
has a convergent subsequence in $H$. The conclusion then follows.
\end{proof}

The following lemma is elementary and we include it here only for completeness.
\begin{lemma}\label{l1}Let $T$ be a symmetric operator with deficiency indices $(n,n)$ such that $n\in \mathbb{N}$. If $(D(T),\|\cdot\|_g)$ is compactly embedded in $H$, then any self-adjoint extension of $T$ has compact resolvent.
\end{lemma}
\begin{proof}Let $\tilde{T}$ be any self-adjoint extension of $T$. Let $\{x_n\}$ be a bounded sequence in $H$ and $x_n=x_n^1+x_n^2$ according to the decomposition $H=\textup{Ran}(T-\mathfrak{i})\oplus_\bot \textup{ker}(T^*+\mathfrak{i})$. Then both $\{x_n^1\}$ and $\{x_n^2\}$ are bounded and
\[(\tilde{T}-\mathfrak{i})^{-1}x_n=(\tilde{T}-\mathfrak{i})^{-1}x_n^1+(\tilde{T}-\mathfrak{i})^{-1}x_n^2=(T-\mathfrak{i})^{-1}x_n^1+(\tilde{T}-\mathfrak{i})^{-1}x_n^2.\]
Since $(T-\mathfrak{i})^{-1}$ is a topological linear isomorphism between $\textup{Ran}(T-\mathfrak{i})$ and $(D(T), \|\cdot\|_g)$ and the latter is compactly embedded in $H$, we see $\{(T-\mathfrak{i})^{-1}x_n^1\}$ has a convergent subsequence in $H$. W.l.g, we still denote such a subsequence by $\{(T-\mathfrak{i})^{-1}x_n^1\}$. On the other side, since the corresponding $\{(\tilde{T}-\mathfrak{i})^{-1}x_n^2\}$ is bounded in $\textup{ker}(T^*-\mathfrak{i})$ which is of dimension $n$, there is a subsequence $\{x_{n_k}\}$ such that $\{(\tilde{T}-\mathfrak{i})^{-1}x_{n_k}^2\}$ is convergent in $H$. Therefore, $\{(\tilde{T}-\mathfrak{i})^{-1}x_{n_k}\}$ is convergent in $H$, i.e., $\tilde{T}$ has compact resolvent.
\end{proof}
The proof clearly depends on the fact that $n$ is finite and thus when $n=\infty$ much effort is needed.
\begin{definition}A symmetric operator $T$ in a Hilbert space $H$ is called simple, if it is not a nontrivial orthogonal sum of a self-adjoint operator and a symmetric operator.
\end{definition}
Given a symmetric operator $T$ with deficiency indices $(n,n)$, we can decompose $T$ into the orthogonal direct sum of its self-adjoint part $T'$ and its simple symmetric part $T''$ whose deficiency indices are $(n,n)$. In particular, $D(T)=D(T')+D(T'')$ and $H=\overline{D(T')}\oplus_\bot \overline{D(T'')}$.
\begin{lemma}Let $T$, $T'$ and $T''$ be as above. Then $T$ has a self-adjoint extension with compact resolvent if and only if $T'$ has compact resolvent and $T''$ has a self-adjoint extension with compact resolvent.
\end{lemma}
\begin{proof}This will be obvious once one realizes that the Cayley transform of any self-adjoint extension $\tilde{T}$ of $T$ is actually the orthogonal direct sum of the Cayley transform of $T'$ and the Cayley transform of a self-adjoint extension of $T''$.
\end{proof}

If $(D(T), \|\cdot\|_g)$ is compactly embedded in $H$, then $D(T')$ as a closed subspace of $(D(T), \|\cdot\|_g)$ is compactly embedded in $\overline{D(T')}$ as well. By an argument similar to the proof of Lemma \ref{l1}, we see that $T'$ has compact resolvent. Therefore, to prove the sufficiency part of Theorem 1.1, we only need to consider the case where $T$ is simple. Furthermore by Lemma \ref{l1}, it suffices to consider the case where $T$ has deficiency indices $(\infty, \infty)$.
\section{Existence of self-adjoint extensions with compact resolvent}\label{existence}
We assume throughout this section that our symmetric operator $T$ in $H$ is simple, has deficiency indices $(\infty,\infty)$, and $(D(T), \|\cdot\|_g)$ is compactly embedded in $H$.

Recall that $\lambda\in \mathbb{C}$ is \emph{of regular type} for a symmetric operator $T$ if there is a constant $c_\lambda>0$ such that
\begin{equation}\|(T-\lambda)x\|\geq c_\lambda \|x\|\label{eq2}\end{equation}
for any $x\in D(T)$. Denote the set of $\lambda$ of regular type for $T$ by $\Theta(T)$. $\Theta(T)$ is open in $\mathbb{C}$ and basically the upper and lower half planes $\mathbb{C}_+\cup \mathbb{C}_-\subseteq \Theta(T)$. If furthermore $\Theta(T)=\mathbb{C}$, $T$ is called \emph{regular}. Regular operators are necessarily simple.

\begin{proposition}\label{l2}If $(D(T), \|\cdot\|_g)$ is compactly embedded in $H$, then $T$ is regular.
\end{proposition}
\begin{proof}We have to prove that $\mathbb{R}\subseteq \Theta(T)$. We shall only prove $0\in \Theta(T)$ and for other $\lambda\in \mathbb{R}$ it suffices to replace $T$ with $T-\lambda$.

$0\in \Theta(T)$ will follow immediately from Closed Graph Theorem if we can show $\textup{Ran}T$ is closed in $H$. Let $T=U_T|T|$ be the polar decomposition of $T$. See \cite[\S.~3.2, \S~7.1]{schmudgen2012unbounded} for the basics of this decomposition for unbounded operators. We note that (1) $|T|=(T^*T)^{1/2}$ is a densely defined self-adjoint operator whose domain $D(|T|)=D(T)$ and in particular $\|Tx\|=\||T|x\|$ for each $x\in D(T)$,  and (2) $U_T$ is a partial isometry whose initial (resp.~final) space is $\overline{\textup{Ran}|T|}$ (resp.~$\overline{\textup{Ran}T}$). Since $U_T$ transforms $\textup{Ran}|T|$ into $\textup{Ran}T$, $\textup{Ran}T$ is closed if and only if $\textup{Ran}|T|$ is closed.

Obviously, $\textup{ker}|T|=\textup{ker}T$. Since $T$ is simple, $\textup{ker}|T|=0$. The number $0$ lies either in $\sigma_c(|T|)$ or in $\rho(|T|)$. On the other side, since $|T|$ is positive-definite, $Id+|T|$ is invertible and consequently $(Id+|T|)^{-1}$ is a bounded operator from $H$ to $D(|T|)$. $(Id+|T|)^{-1}$ has to be a compact operator on $H$ because $D(|T|)=D(T)\hookrightarrow H$ is compact w.r.t. the graph norm. Let $\{\mu_i\}$ be the eigenvalues of $(Id+|T|)^{-1}$. Then $\sigma(|T|)=\{\frac{1}{\mu_i}-1\}$ and therefore $0\in \rho(|T|)$. Hence $\textup{Ran}|T|=H$. The proof then follows.
\end{proof}

We need the notion of boundary triplet to carry out a further analysis of $T$ and refer the reader to \cite{behrndt2020boundary,schmudgen2012unbounded} for a detailed presentation for it. Associated to a boundary triplet $(G, \Gamma_0, \Gamma_1)$ of $T$ is the important concept of Weyl function. Denote the self-adjoint extension $T^*|_{\Gamma_0^{-1}(0)}$ by $\mathrm{T}_0$. Then the $\gamma$-field $\gamma(\lambda)$ associated to the boundary triplet is a $\mathbb{B}(G, H)$-valued analytic function defined on $\rho(\mathrm{T}_0)$: For any $\varphi\in G$, $\gamma(\lambda)\varphi$ is the unique solution $x$ of the abstract boundary value problem
\[\left\{
\begin{array}{ll}
T^*x=\lambda x,\\
\Gamma_0x=\varphi.
\end{array}
\right.\]
Then the Weyl function $M(\lambda)$ as a $\mathbb{B}(G)$-valued analytic function is defined to be $M(\lambda)\varphi:=\Gamma_1\gamma(\lambda)\varphi$. $M(\lambda)$ is an operator-valued Nevanlinna function, i.e., for $\lambda \in \mathbb{C}\setminus \mathbb{R}$, $M(\bar{\lambda})=M(\lambda)^*$ and $\Im M(\lambda)/\Im \lambda>0$. A basic formula we shall use later is
\begin{equation}M'(\lambda)=\frac{d}{d\lambda}M(\lambda)=\gamma(\bar{\lambda})^*\gamma(\lambda),\label{equation}\end{equation}
which can be found in \cite[Prop.~14.15 (iv)]{schmudgen2012unbounded}.

A geometric interpretation of $M(\lambda)$ goes as follows. The quotient Hilbert space $\mathrm{B}_T:=D(T^*)/D(T)$ has an indefinite Hermitian inner product defined by
\[\{[x],[y]\}=-\mathfrak{i}[(T^*x,y)_H-(x,T^*y)_H].\]
The fact that $T$ has deficiency indices $(\infty, \infty)$ can be rephrased as that this indefinite Hermitian inner product has signature $(\infty, \infty)$. Similarly, on $\mathbb{G}:=G\oplus_\bot G$, we also have the canonical indefinite Hermitian inner product
\[\{(\varphi_0, \varphi_1), (\psi_0, \psi_1)\}_\mathbb{G}=-\mathfrak{i}[(\varphi_1, \psi_0)_G-(\varphi_0, \psi_1)_G]\]
where $(\varphi_0, \varphi_1)$, $(\psi_0, \psi_1)\in \mathbb{G}$. This indefinite Hermitian inner product also has signature $(\infty, \infty)$. Note that $\Gamma_0^{-1}(0)\cap \Gamma_1^{-1}(0)=D(T)$. Then that $(G, \Gamma_0, \Gamma_1)$ is a boundary triplet means precisely that $(\Gamma_0, \Gamma_1)$ is essentially an isometry from $(\mathrm{B}_T, \{\cdot, \cdot\})$ to $(\mathbb{G}, \{\cdot, \cdot\}_\mathbb{G})$. Denote the image of $\textup{ker}(T^*-\lambda)$ in $\mathrm{B}_T$ by $\mathrm{W}_\lambda$. The isometry transforms $\mathrm{W}_\lambda$ into a closed subspace $\mathrm{M}_\lambda$ of $\mathbb{G}$. In particular, if $\lambda\in \rho(\mathrm{T}_0)$, $\mathrm{M}_\lambda$ can be represented in terms of the operator $M(\lambda)$, i.e.,
\[\mathrm{M}_\lambda=\{(\varphi, M(\lambda)\varphi)\in \mathbb{G}|\varphi\in G\}.\]
One should note that unlike $M(\lambda)$, $\mathrm{M}_\lambda$ makes sense for all $\lambda\in \mathbb{C}$, though for $\lambda\in \sigma(\mathrm{T}_0)$ the above representation no longer holds.

A closed subspace $W$ of $\mathrm{B}_T$ or $\mathbb{G}$ is called \emph{isotropic} if the indefinite Hermitian inner product is trivial when restricted on $W$. $W$ is called \emph{Lagrangian} if it is isotropic and cannot be contained properly in a larger isotropic subspace. All self-adjoint extensions of $T$ are parameterized precisely by Lagrangians in $\mathrm{B}_T$ or $\mathbb{G}$, e.g., in $\mathrm{B}_T$ the Lagrangian corresponding to $\mathrm{T}_0$ is $D(\mathrm{T}_0)/D(T)$ and via the above isometry the corresponding Lagrangian in $\mathbb{G}$ is $L_0:=0\oplus G$. More generally any closed extension $\tilde{T}$ corresponds in this way to a closed subspace $W$ in $\mathrm{B}_T$ or $\mathbb{G}$ and $\tilde{T}^*$ corresponds to the orthogonal complement of $W$ w.r.t. the indefinite Hermitian inner product. See \cite[Lemma 14.6, Prop.~14.7]{schmudgen2012unbounded} for details of these claims.
\begin{proposition}\label{trans}For each $\lambda\in \mathbb{R}$, $\mathrm{M}_\lambda$ is Lagrangian in $\mathbb{G}$. $\lambda\in \rho(\mathrm{T}_0)$ if and only if $L_0$ and $\mathrm{M}_\lambda$ are transversal, or equivalently, (1) $L_0\cap \mathrm{M}_\lambda=0$, and (2) $L_0\oplus \mathrm{M}_\lambda=\mathbb{G}$.
\end{proposition}
\begin{proof}The result was actually proved in \cite{wang2024complex}. For the reader's convenience we present a self-contained proof here.

$\mathrm{M}_\lambda$ is isotropic in $\mathbb{G}$. Indeed, if $(\varphi_0, \varphi_1), (\psi_0, \psi_1)\in \mathrm{M}_\lambda\subseteq \mathbb{G}$, then by definition there are $x,y\in \textup{ker}(T^*-\lambda)$ such that $(\Gamma_0x, \Gamma_1x)=(\varphi_0, \varphi_1)$ and $(\Gamma_0y, \Gamma_1y)=(\psi_0, \psi_1)$. According to the abstract Green's formula, we have
\begin{eqnarray*}(\varphi_1, \psi_0)_G-(\varphi_0, \psi_1)_G&=&(\Gamma_1x,\Gamma_0y)_G-(\Gamma_0x,\Gamma_1y)_G=(T^*x,y)_H-(x,T^*y)_H\\
&=&(\lambda x,y)_H-(x,\lambda y)_H=0.\end{eqnarray*}

If $\mathrm{M}_\lambda\subseteq W$ where $W$ is an isotropic subspace of $\mathbb{G}$, then for any $(\varphi_0, \varphi_1)\in W$, by definition there is a $z\in D(T^*)$ such that $(\Gamma_0 z, \Gamma_1z)=(\varphi_0, \varphi_1)$ and for any $x\in \textup{ker}(T^*-\lambda)$ we have
\begin{eqnarray*}0&=&(\Gamma_1x, \Gamma_0z)_G-(\Gamma_0x, \Gamma_1z)_G=(T^*x,z)_H-(x,T^*z)_H\\
&=&\lambda(x,z)_H-(x,T^*z)_H=-(x, (T^*-\lambda)z)_H,\end{eqnarray*}
implying that $(T^*-\lambda)z\in [\textup{ker}(T^*-\lambda)]^\bot=\textup{Ran}(T-\lambda)$. Thus there is a $z_0\in D(T)$ such that $(T^*-\lambda)z=(T-\lambda)z_0$. Consequently, $z-z_0\in \textup{ker}(T^*-\lambda)$ and $(\varphi_0, \varphi_1)=(\Gamma_0(z-z_0),\Gamma_1(z-z_0))\in \mathrm{M}_\lambda$. This shows $W=\mathrm{M}_\lambda$.

If $\lambda\in \rho(\mathrm{T}_0)$, then $D(T^*)=D(\mathrm{T}_0)\oplus \textup{ker}(T^*-\lambda)$. Indeed we only have to prove the inclusion $D(T^*)\subseteq D(\mathrm{T}_0)\oplus \textup{ker}(T^*-\lambda)$. For any $x\in D(T^*)$, $(\mathrm{T}_0-\lambda)^{-1}(T^*-\lambda)x\in D(\mathrm{T}_0)$. On the other side
\[(T^*-\lambda)(x-(T_0-\lambda)^{-1}(T^*-\lambda)x)=(T^*-\lambda)x-(T^*-\lambda)(T_0-\lambda)^{-1}(T^*-\lambda)x=0,\]
implying that $x-(T_0-\lambda)^{-1}(T^*-\lambda)x\in \textup{ker}(T^*-\lambda)$. Applying $(\Gamma_0, \Gamma_1)$ to both sides of the identity $D(T^*)=D(\mathrm{T}_0)\oplus \textup{ker}(T^*-\lambda)$, we get $\mathbb{G}=L_0+\mathrm{M}_\lambda$. If $0\neq (\varphi_0,\varphi_1)\in L_0\cap \mathrm{M}_\lambda$, then there is an $x\in \textup{ker}(T^*-\lambda)$ such that $(\Gamma_0x,\Gamma_1x)=(\varphi_0,\varphi_1)=(0,\varphi_1)$. This shows $x\in D(\mathrm{T}_0)$ and thus $\lambda$ is an eigenvalue of $\mathrm{T}_0$. This is a contradiction.

Conversely, if $L_0$ and $\mathrm{M}_\lambda$ are transversal, to see $\lambda\in \rho(\mathrm{T}_0)$, we have to prove that for any given $y\in H$, there is a unique $x\in D(\mathrm{T}_0)$ such that $(\mathrm{T}_0-\lambda)x=y$. Since $T$ is simple, $\textup{ker}(T-\lambda)=0$, and thus $\overline{\textup{Ran}(T^*-\lambda)}=[\textup{ker}(T-\lambda)]^\bot=H$. On the other side, since $T$ is regular, $T-\lambda$ has closed range. By the famous Closed Range Theorem, $\textup{Ran}(T^*-\lambda)$ is closed as well. Therefore $\textup{Ran}(T^*-\lambda)=H$ and we can choose $x_0\in D(T^*)$ such that $(T^*-\lambda)x_0=y$. Since $L_0$ and $\mathrm{M}_\lambda$ are transversal in $\mathbb{G}$, we can find $x_1\in D(\mathrm{T}_0)$ and $x_2\in \textup{ker}(T^*-\lambda)$ such that
\[(\Gamma_0x_0,\Gamma_1x_0)=(\Gamma_0x_1,\Gamma_1x_1)+(\Gamma_0x_2,\Gamma_1x_2).\]
We should have $\Gamma_0x_1=0$ and thus $\Gamma_0x_0=\Gamma_0x_2$. As a result $x_0-x_2\in D(\mathrm{T}_0)$. Now
\[(\mathrm{T}_0-\lambda)(x_0-x_2)=(T^*-\lambda)(x_0-x_2)=(T^*-\lambda)x_0=y.\]
That's to say $x_0-x_2$ is our required $x$. As before, $L_0\cap \mathrm{M}_\lambda=0$ means $\lambda$ cannot be an eigenvalue of $\mathrm{T}_0$, from which the uniqueness of $x$ follows.
\end{proof}
\emph{Remark}. Since $\mathrm{M}_\lambda$ is Lagrangian for $\lambda\in \mathbb{R}$, there is a self-adjoint extension $T_\lambda$ of $T$ corresponding to $\mathrm{M}_\lambda$. It is easy to see
\[D(T_\lambda)=D(T)+\textup{ker}(T^*-\lambda).\] $T_\lambda$ is called an extension of Krein type in \cite[\S\S~5.4]{behrndt2020boundary} where $T$ is assumed to be a semi-bounded symmetric relation.
\begin{lemma}\label{reg}For each $\lambda\in \mathbb{R}$, there is a self-adjoint extension $\mathrm{T}_0$ of $T$ such that $\lambda\in \rho(\mathrm{T}_0)$.
\end{lemma}
\begin{proof}Due to the von Neumann decomposition, we can identify $\mathrm{B}_T$ with $\textup{ker}(T^*-\mathfrak{i})\oplus \textup{ker}(T^*+\mathfrak{i})$ and self-adjoint extensions correspond to unitary maps from $\textup{ker}(T^*-\mathfrak{i})$ to $\textup{ker}(T^*+\mathfrak{i})$. If $U$ is the unitary map corresponding to $T_\lambda$, i.e., the Lagrangian in $\mathrm{B}_T$ associated to $T_\lambda$ is the graph of $U$, then it can be checked easily that the graph of $-U$ is another Lagrangian transverse to the graph of $U$. Let $\mathrm{T}_0$ be the self-adjoint extension associated to $-U$. Then by \cite[Thm.~2.5.9]{behrndt2020boundary}, we can find a boundary triplet $(G, \Gamma_0, \Gamma_1)$ such that $\mathrm{T}_0=T^*|_{\Gamma_0^{-1}(0)}$ and $T_\lambda=T^*|_{\Gamma_1^{-1}(0)}$. Since transversality is topologically invariant, we know from Prop.~\ref{trans} that $\lambda\in \rho(\mathrm{T}_0)$.
\end{proof}

Let $W_1, W_2$ be two closed subspaces of infinite dimension and codimension in a Hilbert space $H$ and $P_1, P_2$ the corresponding orthogonal projections onto them. $W_2$ is called a compact perturbation of $W_1$ if $P_2-P_1$ is a compact operator. In this case, $(W_2,W_1^\bot)$ is a Fredholm pair of subspaces in $H$ in the sense that both $\textup{dim}(W_2\cap W_1^\bot)$ and $\textup{codim}(W_2+W_1^\bot)$ are finite. The number
\[\textup{dim}(W_2,W_1):=\textup{dim}(W_2\cap W_1^\bot)-\textup{codim}(W_2+W_1^\bot)\]
is called the \emph{relative dimension} of $W_2$ w.r.t. $W_1$ \cite[\S~3]{abbondandolo2009infinite}. If a given closed subspace $W\subseteq H$ is of infinite dimension and codimension, we use $Gr_c(W, H)$ to denote the Grassmannian of compact perturbations of $W$ in $H$. $Gr_c(W, H)$ is a complex Banach manifold and has many topological components $Gr_{c,d}(W,H)$, each of which is indexed by the relative dimension $d$ w.r.t. $W$. Let $\mathbb{GL}_c(H)$ be the group of invertible elements in $\mathbb{B}(H)$ which are of the form $Id+K$ for a compact operator $K$.
\begin{lemma}\label{topo}$\mathbb{GL}_c(H)$ acts naturally on $Gr_c(W,H)$ and preserves each $Gr_{c,d}(W,H)$. In particular, the $\mathbb{GL}_c(H)$-action on $Gr_{c,d}(W,H)$ is transitive.
\end{lemma}
\begin{proof}The claims can be found in \cite[\S~3]{abbondandolo2009infinite}. In particular, the transitivity of the group action can be proved by a slight modification of the proof of \cite[Lemma 15.10]{booss1993elliptic}.
\end{proof}
More information on $Gr_c(W, H)$ and the $\mathbb{GL}_c(H)$-action on it can be found in \cite{abbondandolo2009infinite}. The following lemma shows why the group action of $\mathbb{GL}_c(H)$ on $Gr_c(W, H)$ is relevant here.
\begin{lemma}If $(D(T), \|\cdot\|_g)$ is compactly embedded in $H$, then $T$ has a self-adjoint extension with compact resolvent if and only if $\textup{ker}(T^*+\mathfrak{i})$ is a compact perturbation of $\textup{ker}(T^*-\mathfrak{i})$ and the relative dimension is zero.
\end{lemma}
\begin{proof}For our convenience, we use $W_\pm$ to denote $\textup{ker}(T^*\mp \mathfrak{i})$ and let $P_\pm$ denote the corresponding orthogonal projections.

\emph{The necessity part}. If $\mathrm{T}_0$ is a self-adjoint extension of $T$ with compact resolvent, then the Cayley transform $C(\mathrm{T}_0)\in \mathbb{GL}_c(H)$ by definition. Clearly $C(\mathrm{T}_0)$ sends $W_+$ into $W_-$, implying due to Lemma \ref{topo} that $W_-$ lies in $Gr_{c,0}(W, H)$ for $W=W_+$.

\emph{The sufficiency part}. We can choose $g\in \mathbb{GL}_c(H)$ such that $g\cdot W_+=W_-$. Let $gP_+=V\rho$ be the polar decomposition of $gP_+$, where $\rho$ is nonnegative, $\textup{ker}\rho=\textup{ker}(gP_+)=W_+^\bot$ and $V$ is a partial isometry whose initial (resp. final) space is $\overline{\textup{Ran}(\rho)}=W_+$ (resp. $W_-$). Then $V|_{W_+}$ is an isometry from $W_+$ to $W_-$. Since $g\in \mathbb{GL}_c(H)$, so is $g^*$. Thus
\[\rho^2=P_+g^*gP_+=P_++P_+KP_+\]
where $K$ is self-adjoint and compact. Thus $\rho|_{W_+}$ is Fredholm as an operator on $W_+$. Since $\textup{ker}\rho=W_+^\bot$, we see that $\rho|_{W_+}$ is invertible and lies in $\mathbb{GL}_c(W_+)$. We claim that the inverse Cayley transform of the unitary operator $U:=C(T)\oplus_\bot V|_{W_+}$ is a self-adjoint extension of $T$ with compact resolvent, i.e., $U-Id$ is compact. Indeed, let $\{z_n\}$ be a bounded sequence in $H$ and $z_n=x_n+y_n$ according to the decomposition $H=W_+\oplus_\bot W_+^\bot$. Then by definition
\begin{eqnarray*}Uz_n-z_n&=&((T-\mathfrak{i})(T+\mathfrak{i})^{-1}y_n-y_n)+(Vx_n-x_n)\\
&=&-2\mathfrak{i}(T+\mathfrak{i})^{-1}y_n+V\rho (\rho|_{W_+})^{-1}x_n-x_n\\
&=&-2\mathfrak{i}(T+\mathfrak{i})^{-1}y_n+g(\rho|_{W_+})^{-1}x_n-x_n.
\end{eqnarray*}
Note that $T+\mathfrak{i}$ is a topological linear isomorphism between $(D(T), \|\cdot\|_g)$ and $\textup{Ran}(T+\mathfrak{i})$. Since $(D(T), \|\cdot\|_g))$ is compactly embedded in $H$, we see that $\{Uz_n-z_n\}$ has a convergent subsequence in $H$.
\end{proof}
\begin{lemma}If $(D(T), \|\cdot\|_g)$ is compactly embedded in $H$, then $\textup{ker}(T^*+\mathfrak{i})$ is a compact perturbation of $\textup{ker}(T^*-\mathfrak{i})$.
\end{lemma}
\begin{proof}We continue to use the notation in the previous lemma. We extend $C(T)$ to a partial isometry $U_0$ with initial subspace $W_+^\bot=\textup{Ran}(T+\mathfrak{i})$ and final subspace $\textup{Ran}(T-\mathfrak{i})$. Then
\[Id-P_+=U_0^*U_0,\quad Id-P_-=U_0U_0^*.\]
It's easy to see the compactness of the embedding $D(T)\hookrightarrow H$ is equivalent to that $K_1:=U_0(Id-P_+)-(Id-P_+)$ is compact. Note that
\[K_1=U_0U_0^*U_0-(Id-P_+)=(Id-P_-)U_0-(Id-P_+)\]
and consequently $K_1^*=U_0^*(Id-P_-)-(Id-P_+)$. We see
\[K_1U_0^*=U_0U_0^*U_0U_0^*-U_0^*U_0U_0^*=Id-P_--U_0^*(Id-P_-)\]
and finally find that
\[P_+-P_-=(Id-P_-)-(Id-P_+)=K_1U_0^*+K_1^*.\]
The lemma then follows.
\end{proof}
Therefore, to prove Thm.~\ref{Th1}, we only have to prove that $\textup{dim}(W_-,W_+)=0$ when $D(T)\hookrightarrow H$ is compact. For each $\lambda\in \mathbb{C}$, denote $\textup{ker}(T^*-\lambda)$ by $\mathrm{K}_\lambda$. Note that for $\lambda\in \mathbb{C}_+$, $C_\lambda(T):=(T-\lambda)(T-\bar{\lambda})^{-1}$ is a unitary map from $\textup{Ran}(T-\bar{\lambda})$ to $\textup{Ran}(T-\lambda)$. Then along the same line, for each $\lambda\in \mathbb{C}_+$ we can prove that $\mathrm{K}_{\bar{\lambda}}$ is a compact perturbation of $\mathrm{K}_\lambda$. Define $n(\lambda):=\textup{dim}(\mathrm{K}_{\bar{\lambda}},\mathrm{K}_\lambda)$ for $\lambda\in \mathbb{C}_+$.
\begin{lemma}If $(D(T), \|\cdot\|_g)$ is compactly embedded in $H$ and $M(\lambda)$ is the Weyl function associated to a boundary triplet $(G, \Gamma_0, \Gamma_1)$, then $M'(\lambda)$ is Fredholm for each $\lambda\in \mathbb{C}_+$ and
$n(\lambda)=-\textup{ind}(M'(\lambda))$.
\end{lemma}
\begin{proof}Let $\gamma(\lambda)$ be the $\gamma$-field associated to the given boundary triplet. Note that $\mathrm{K}_\lambda=\gamma(\lambda)G$ and $\mathrm{K}_{\bar{\lambda}}=\gamma(\bar{\lambda})G$. Then
\begin{eqnarray*}\dim(\mathrm{K}_{\bar{\lambda}}\cap \mathrm{K}_\lambda^\bot)&=&\textup{dim}\{\varphi\in G|(\gamma(\bar{\lambda})\varphi, \gamma(\lambda)\psi)_G=0,\ , \forall \psi\in G\}\\
&=&\dim\textup{ker}(\gamma(\lambda)^*\gamma(\bar{\lambda})).\end{eqnarray*}
Similarly,
\[\textup{codim}(\mathrm{K}_{\bar{\lambda}}+\mathrm{K}_\lambda^\bot)=\dim(\mathrm{K}_\lambda\cap \mathrm{K}_{\bar{\lambda}}^\bot)=\dim\textup{ker}(\gamma(\bar{\lambda})^*\gamma(\lambda)).\]
Note that due to Eq.~(\ref{equation}), $\gamma(\bar{\lambda})^*\gamma(\lambda)=M'(\lambda)$ and $\gamma(\lambda)^*\gamma(\bar{\lambda})=M'(\bar{\lambda})=(M'(\lambda))^*$. What remained is to show that $\gamma(\bar{\lambda})^*\gamma(\lambda)G=\gamma(\bar{\lambda})^*\mathrm{K}_\lambda$ is closed in $G$. Let $P_\lambda$ be the projection onto $\mathrm{K}_\lambda$. Since $\mathrm{K}_{\bar{\lambda}}$ is a compact perturbation of $\mathrm{K}_\lambda$, $(P_\lambda, P_{\bar{\lambda}})$ is a \emph{Fredholm pair of projections} or equivalently $P_{\bar{\lambda}}P_\lambda$ as a map from $\mathrm{K}_\lambda$ to $\mathrm{K}_{\bar{\lambda}}$ is necessarily Fredholm \cite[\S~5.3]{doll2023spectral}. Thus $P_{\bar{\lambda}}\mathrm{K}_\lambda$ is closed in $\mathrm{K}_{\bar{\lambda}}$. Note that $\gamma(\bar{\lambda})^*\mathrm{K}_\lambda=\gamma(\bar{\lambda})^*P_{\bar{\lambda}} \mathrm{K}_\lambda$ and $\gamma(\bar{\lambda})^*$ is a topological linear isomorphism from $\mathrm{K}_{\bar{\lambda}}$ to $G$. We see that $\gamma(\bar{\lambda})^*\mathrm{K}_\lambda$ is closed in $G$.
\end{proof}
\emph{Remark}. A byproduct of the proof is the general fact that $\dim\textup{ker}(M'(\lambda))$ for $\lambda\in \mathbb{C}_+$ is independent of the choice of boundary triplets.

To complete the proof of Thm.~\ref{Th1}, we can choose a boundary triplet such that $M(\lambda)$ is analytic at $0\in \mathbb{C}$. Indeed by the proof of Lemma \ref{reg}, we can further require that $0\in \rho(\mathrm{T}_0)$ and thus $\rho(\mathrm{T}_0)$ is connected. By \cite[Coro.~2.3.8]{behrndt2020boundary}, $M'(0)$ is self-adjoint and invertible. This of course implies that $M'(\lambda_0)$ is invertible as well for a $\lambda_0\in \mathbb{C}_+$ sufficiently close to 0. By topological invariance of $\textup{ind}(M'(\lambda))$ and the connectedness of $\mathbb{C}_+$, we see
\[\textup{dim}(W_-,W_+)=\textup{dim}(\mathrm{K}_{-\mathfrak{i}},\mathrm{K}_{\mathfrak{i}})=n(\mathfrak{i})=n(\lambda_0)=-\textup{ind}(M'(\lambda_0))=0.\]
\emph{Remark.} Though the existence of a self-adjoint extension with compact resolvent can be guaranteed by the compactness of the embedding $D(T)\hookrightarrow H$, it doesn't mean that under this condition a certain "natural" self-adjoint extension has compact resolvent. Recently M. Malamud has shown that if additionally $T$ is semi-bounded, the Friedrichs' extension $T_F$ does not necessarily have compact resolvent \cite{malamud2023birman}.
\section{Parameterization of all self-adjoint extensions with compact resolvent}\label{para}
Now we assume that $(D(T),\|\cdot\|_g)$ is compactly embedded in $H$. Our basic goal in this section is to find all self-adjoint extensions of $T$ with compact resolvent. W.l.g., we further assume that $T$ is simple and has deficiency indices $(\infty, \infty)$. We can fix a self-adjoint extension $\mathrm{T}_0$ with compact resolvent and choose a boundary triplet $(G, \Gamma_0, \Gamma_1)$
such that $\mathrm{T}_0=T^*|_{\Gamma_0^{-1}(0)}$.

We have to establish a variant of Krein type resolvent formula. To do this, we define maps $\Gamma_\pm: D(T^*)\rightarrow G$ via $\Gamma_\pm=\frac{1}{\sqrt{2}}(\Gamma_1\pm \mathfrak{i}\Gamma_0)$. Then the abstract Green's formula now reads
\[(T^*x,y)_H-(x,T^*y)_H=\mathfrak{i}(\Gamma_+x,\Gamma_+y)_G-\mathfrak{i}(\Gamma_-x,\Gamma_-y)_G.\]
In terms of $\Gamma_\pm$, self-adjoint extensions are parameterized precisely by unitary operators on $G$, i.e., for a unitary operator $U\in \mathbb{U}(G)$, the corresponding self-adjoint extension $T_U$ has
\[\{x\in D(T^*)|\Gamma_-x=U\Gamma_+x\}\]
as its domain. In particular, the unitary operator $Id$ corresponds to $\mathrm{T}_0$ and $-Id$ corresponds to $\mathrm{T}_1:=T^*|_{\Gamma_1^{-1}(0)}$. Associated to $\Gamma_\pm$ are the two (non-self-adjoint) extensions $\mathrm{T}_\pm:=T^*|_{\Gamma_\pm^{-1}(0)}$. It is not hard to see $\mathrm{T}_-=\mathrm{T}_+^*$.
\begin{lemma}\label{plus}$\mathbb{C}_\pm \subseteq \rho(\mathrm{T}_\pm)$.
\end{lemma}
\begin{proof}This is a direct conclusion of \cite[Coro.~1.6.5]{behrndt2020boundary}, but for the convenience of the reader we provide an independent proof here.
Let $\lambda=a+\mathfrak{i}b\in \mathbb{C}_+$ where $a,b\in \mathbb{R}$. Then for each $x\in D(\mathrm{T}_+)$ we have
\begin{eqnarray*}\|(\mathrm{T}_+-\lambda)x\|^2&=&(\mathrm{T}_+x-ax-\mathfrak{i}bx, \mathrm{T}_+x-ax-\mathfrak{i}bx)_H\\
&=&\|(\mathrm{T}_+-a)x\|^2+b^2\|x\|^2+\mathfrak{i}b[(\mathrm{T}_+x,x)_H-(x,\mathrm{T}_+x)_H]\\
&=&\|(\mathrm{T}_+-a)x\|^2+b^2\|x\|^2+b\|\Gamma_-x\|^2,\end{eqnarray*}
where in the last line the abstract Green's formula is used. Thus $\|(\mathrm{T}_+-\lambda)x\|\geq b\|x\|$. This shows $\lambda\notin \sigma_p(\mathrm{T}_+)\cup \sigma_c(\mathrm{T}_+)$. Similarly, $\lambda\notin \sigma_p(\mathrm{T}_-)\cup \sigma_c(\mathrm{T}_-)$ for any $\lambda\in \mathbb{C}_-$. In fact $\lambda\in \mathbb{C}_+$ cannot lie in $\sigma_r(\mathrm{T}_+)$ too because otherwise $\bar{\lambda}$ ought to lie in $\sigma_p(\mathrm{T}_-)$. This shows $\mathbb{C}_+\subseteq \rho(\mathrm{T}_+)$. The claim for $\mathrm{T}_-$ follows similarly.
\end{proof}
\begin{lemma}Fix $\lambda\in \mathbb{C}_+$. For each $\varphi\in G$, there is a unique $x\in H$ such that \[\left\{
\begin{array}{ll}
T^*x=\lambda x,\\
\Gamma_+x=\varphi.
\end{array}
\right.\]
\end{lemma}
\begin{proof}The uniqueness follows immediately from Lemma \ref{plus}. To see the existence, we first note that $\Gamma_+ D(T^*)=G$. On the other side, as in the proof of Prop.~\ref{trans} we have
\[D(T^*)=D(\mathrm{T}_+)+\textup{ker}(T^*-\lambda).\] Thus we see that $\Gamma_+(D(T^*))=\Gamma_+(\textup{ker}(T^*-\lambda))$. The claim then follows.
\end{proof}
Thus we have a $\mathbb{B}(G, H)$-valued analytic function $\gamma_+(\lambda)$ such that $\gamma_+(\lambda)\varphi$ is the solution $x$ associated to $\varphi\in G$. $\gamma_+(\lambda)$ is actually a topological linear isomorphism between $G$ and $\textup{ker}(T^*-\lambda)$. Now we can also define a $\mathbb{B}(G)$-valued analytic function $B(\lambda)$ by $B(\lambda)\varphi:=\Gamma_-\gamma_+(\lambda)\varphi$. An easy computation shows $B(\lambda)=(M(\lambda)-\mathfrak{i})(M(\lambda)+\mathfrak{i})^{-1}$ for $\lambda\in \mathbb{C}_+$ and as a result $\|B(\lambda)\|<1$. In the same way, replacing $\mathrm{T}_+$ with $\mathrm{T}_-$, $\mathbb{C}_+$ with $\mathbb{C}_-$ and $\Gamma_+$ with $\Gamma_-$, we obtain a $\mathbb{B}(G, H)$-valued analytic function $\gamma_-(\lambda)$ on $\mathbb{C}_-$, which maps $G$ onto $\textup{ker}(T^*-\lambda)$.
\begin{lemma}\label{gama}For $\lambda\in \mathbb{C}_+$, $\gamma_-(\bar{\lambda})^*=-\mathfrak{i}\Gamma_-(\mathrm{T}_+-\lambda)^{-1}$.
\end{lemma}
\begin{proof}For any $x\in H$, let $y=(\mathrm{T}_+-\lambda)^{-1}x\in D(\mathrm{T}_+)$. Then for any $\varphi\in G$, we have
\begin{eqnarray*}(\gamma_-(\bar{\lambda})^*(\mathrm{T}_+-\lambda)y,\varphi)_G&=&((\mathrm{T}_+-\lambda)y, \gamma_-(\bar{\lambda})\varphi)_H\\&=&(\mathrm{T}_+y,\gamma_-(\bar{\lambda})\varphi)_H-(y,\bar{\lambda}\gamma_-(\bar{\lambda})\varphi)_H\\
&=&(T^*y,\gamma_-(\bar{\lambda})\varphi)_H-(y, T^*\gamma_-(\bar{\lambda})\varphi)_H\\
&=&\mathfrak{i}(\Gamma_+y, \Gamma_+\gamma_-(\bar{\lambda})\varphi)_G-\mathfrak{i}(\Gamma_-y, \Gamma_-\gamma_-(\bar{\lambda})\varphi)_G\\
&=&-\mathfrak{i}(\Gamma_-y, \varphi)_G.
\end{eqnarray*}
The formula then follows.
\end{proof}
\begin{lemma}\label{domain}For $\lambda\in \mathbb{C}_+$, in terms of $\mathrm{T}_+$, $\gamma_+(\lambda)$ and $\gamma_-(\bar{\lambda})$, the domain of the self-adjoint extension $T_U$ associated to the unitary operator $U\in \mathbb{U}(G)$ can be characterized in the following way:
\begin{eqnarray*}
D(T_U)&=&\{x=(\mathrm{T}_+-\lambda)^{-1}(y+w)+\gamma_+(\lambda)\varphi|\varphi\in G, w\in \textup{ker}(T^*-\bar{\lambda}),\\
&y&\in \textup{Ran}(T-\lambda), \mathfrak{i}\gamma_-(\bar{\lambda})^*w=(U-B(\lambda))\varphi\}.
\end{eqnarray*}
\end{lemma}
\begin{proof}Denote the RHS by $S$. Due to the decomposition $D(T^*)=D(\mathrm{T}_+)\oplus\textup{ker}(T^*-\lambda)$, if $x\in D(T_U)$, then $x=x_\lambda+z_\lambda$ where $x_\lambda\in D(\mathrm{T}_+)$ and $z_\lambda\in \textup{ker}(T^*-\lambda)$. Since $H=\textup{Ran}(T-\lambda)\oplus_\bot \textup{ker}(T^*-\bar{\lambda})$, there are $y\in \textup{Ran}(T-\lambda)$ and $w\in \textup{ker}(T^*-\bar{\lambda})$ such that $(\mathrm{T}_+-\lambda)x_\lambda=y+w$. Set $y=(T-\lambda)v$ for $v\in D(T)$. Then we have $(\mathrm{T}_+-\lambda)(x_\lambda-v)=w$, implying
$x_\lambda=v+(\mathrm{T}_+-\lambda)^{-1}w$. Thus
\[\Gamma_-x_\lambda=\Gamma_-(\mathrm{T}_+-\lambda)^{-1}w=\mathfrak{i}\gamma_-(\bar{\lambda})^*w\]
due to Lemma \ref{gama}. Besides, by definition we also have $\Gamma_-z_\lambda=B(\lambda)\Gamma_+z_\lambda$. Since $\Gamma_-x=U\Gamma_+x$, we obtain
$$\Gamma_-x_\lambda+\Gamma_-z_\lambda=U\Gamma_+z_\lambda$$
 where the fact $\Gamma_+x_\lambda=0$ is used. Combining these facts together and setting $\varphi=\Gamma_+z_\lambda$, we then have $D(T_U)\subseteq S$. The inclusion $S\subseteq D(T_U)$ can be checked directly.
\end{proof}
Note that if $x=(\mathrm{T}_+-\lambda)^{-1}(y+w)+\gamma_+(\lambda)\varphi\in D(T_U)$ as above, then
\[(T_U-\lambda)x=(T^*-\lambda)x=y+w.\]
\begin{proposition}\label{res}For $\lambda\in \mathbb{C}_+$, we have the formula
\begin{equation}(T_U-\lambda)^{-1}-(\mathrm{T}_+-\lambda)^{-1}=\mathfrak{i}\gamma_+(\lambda)(U-B(\lambda))^{-1}\gamma_-(\bar{\lambda})^*.\label{e2}\end{equation}
\end{proposition}
\begin{proof}For $z\in H$, set $x=(T_U-\lambda)^{-1}z\in D(T_U)$. Then due to Lemma \ref{domain},
\begin{eqnarray*}(T_U-\lambda)^{-1}z&=&x=(\mathrm{T}_+-\lambda)^{-1}(T_U-\lambda)x+\gamma_+(\lambda)\Gamma_+x\\
&=&(\mathrm{T}_+-\lambda)^{-1}z+\gamma_+(\lambda)\Gamma_+x.\end{eqnarray*}
Again due to Lemma \ref{domain},
\[\Gamma_+x=\mathfrak{i}(U-B(\lambda))^{-1}\gamma_-(\bar{\lambda})^*w=\mathfrak{i}(U-B(\lambda))^{-1}\gamma_-(\bar{\lambda})^*(y+w)\]
 where the fact $\gamma_-(\bar{\lambda})^*y=0$ is used. Since $y+w=(T_U-\lambda)x=z$, the resolvent formula thus follows.
\end{proof}
\textbf{The final proof of Thm.~\ref{thm2}}

\begin{proof}This is an easy consequence of Prop.~ \ref{res}. With the boundary triplet $(G, \Gamma_0, \Gamma_1)$, $\mathrm{T}_0$ is parameterized by the unitary operator $Id$. If $U$ is another unitary operator on $G$, then for $\lambda=\mathfrak{i}$
\begin{eqnarray*}(T_U-\mathfrak{i})^{-1}-(\mathrm{T}_0-\mathrm{i})^{-1}&=&\mathfrak{i}\gamma_+(\mathfrak{i})[(U-B(\mathfrak{i}))^{-1}-(Id-B(\mathfrak{i}))^{-1}]\gamma_-(-\mathfrak{i})^*\\
&=&\mathfrak{i}\gamma_+(\mathfrak{i})(U-B(\mathfrak{i}))^{-1}(Id-U)(Id-B(\mathfrak{i}))^{-1}\gamma_-(-\mathfrak{i})^*.
\end{eqnarray*}
Thus $(T_U-\mathfrak{i})^{-1}$ is compact if $U-Id$ is. Conversely, if $T_U$ has compact resolvent, then
\[\gamma_+(\mathfrak{i})^*\gamma_+(\mathfrak{i})(U-B(\mathfrak{i}))^{-1}(Id-U)(Id-B(\mathfrak{i}))^{-1}\gamma_-(-\mathfrak{i})^*\gamma_-(-\mathfrak{i})\]
is compact. Note that both $\gamma_+(\mathfrak{i})^*\gamma_+(\mathfrak{i})$ and $\gamma_-(-\mathfrak{i})^*\gamma_-(-\mathfrak{i})$ are invertible on $G$. We thus see $U-Id$ must be compact either.
\end{proof}

Thus self-adjoint extensions of $T$ with compact resolvent only occupy a small portion of all self-adjoint extensions. This is in sharp contrast with the case with finite deficiency indices $(n,n)$. One basic part of the study of symmetric operators with deficiency indices $(\infty, \infty)$ is to single out those "good" boundary conditions in a certain sense, e.g., \emph{regular} self-adjoint boundary conditions for the Laplacian $\triangle$ on a bounded smooth domain in $\mathbb{R}^m$. Self-adjoint extensions without compact resolvent seem to be "bad". However, such extensions are unavoidable in the following sense.
\begin{proposition}For any unitary operator $U$ close sufficiently to $-Id$, $T_U$ cannot have compact resolvent.
\end{proposition}
\begin{proof}Due to the well-known Kuiper theorem \cite{kuiper1965homotopy}, the unitary group $\mathbb{U}(G)$ is a contractible Banach-Lie group and any $U\in \mathbb{U}(G)$ close enough to $Id$ is of the form $e^{\mathfrak{i}A}$ for some bounded self-adjoint operator $A$. Let $A\in \mathbb{B}(G)$ be self-adjoint such that $\|A\|<\ln 3$. Then
\[\|\frac{1}{2}(e^{\mathfrak{i}A}+Id)-Id\|=\frac{1}{2}\|e^{\mathfrak{i}A}-Id\|\leq \frac{1}{2}(e^{\|A\|}-1)< 1,\]
and consequently $\frac{1}{2}(e^{\mathfrak{i}A}+Id)$ is invertible. Set $U=-e^{\mathrm{i}A}$. Then $U-Id=-e^{\mathfrak{i}A}-Id$ is invertible and cannot be compact. The claim then follows from Thm.~\ref{thm2}.
\end{proof}
This surely means the self-adjoint extension $T^*|_{\Gamma_1^{-1}(0)}$ and those close to it can never have compact resolvent.

\end{document}